\documentclass [12pt,reqno] {article}

\usepackage [usenames] {color}
\usepackage {amssymb}
\usepackage {graphicx}
\usepackage {amscd}

\usepackage [colorlinks=true,
linkcolor=webgreen,
filecolor=webbrown,
citecolor=webgreen] {hyperref}

\definecolor {webgreen} {rgb} {0,.5,0}
\definecolor {webbrown} {rgb} {.6,0,0}

\usepackage {color}
\usepackage {fullpage}
\usepackage {float}

\usepackage {graphics,amsmath,amssymb}
\usepackage{amsthm}
\usepackage {amsfonts}
\usepackage {latexsym}
\usepackage {epsf}

\newcommand{\seqnum}[1]
{\href{http://www.oeis.org/#1}{\underline{#1}}}

\usepackage{url}
\newtheorem{thm}{Theorem}

\newtheorem{prop}{Proposition}
\newtheorem{cor}{Corollary}
\newtheorem{prob}{Open Problem}
\newtheorem{conj}{Conjecture}
\newtheorem*{ident*}{Identity}

\begin{document}

~\\[0.4in]
\begin {center}
\vskip 1cm{\LARGE\bf 
Ramanujan Primes: \\
\vskip .1in
Bounds, Runs, Twins, and Gaps 
}

\vskip 1cm
\large
Jonathan~Sondow\\
209~West~97th~Street\\
New York,~NY~~10025\\
USA\\
\href {mailto:jsondow@alumni.princeton.edu} {\tt jsondow@alumni.princeton.edu} \\
\ \\
John~W.~Nicholson\\
P.~O.~Box~2423\\
Arlington,~TX~~76004\\
USA\\
\href {mailto:reddwarf2956@yahoo.com} {\tt reddwarf2956@yahoo.com} \\
\ \\
Tony~D.~Noe\\
14025~NW~Harvest~Lane\\
Portland,~OR~~97229\\
USA\\
\href {mailto:noe@sspectra.com} {\tt noe@sspectra.com} \\
\end {center}

\vskip .2in

\begin{abstract}
The $n$th Ramanujan prime is the smallest positive integer $R_n$	such that if $x \ge R_n$, then the interval $\left(\frac12x,x\right]$ contains at least $n$ primes. We sharpen Laishram's theorem that $R_n < p_{3n}$ by proving that the maximum of $R_n/p_{3n}$ is $R_5/p_{15} = 41/47$. We give statistics on the length of the longest run of Ramanujan primes among all primes $p<10^n$, for $n\le9$. We prove that if an upper twin prime is Ramanujan, then so is the lower; a table gives the number of twin primes below $10^n$ of three types. Finally, we relate runs of Ramanujan primes to prime gaps. Along the way we state several conjectures and open problems. The Appendix explains Noe's fast algorithm for computing $R_1,R_2,\dotsc,R_n$.\\
\end{abstract}


\section {Introduction} \label{SEC: introd}

For $n\ge1$, the $n$th \emph{Ramanujan prime} is defined as the smallest positive integer $R_n$ with the property that for any $x \ge R_n$, there are at least $n$ primes $p$ with $\frac12x<p\le x$. By its minimality, $R_n$~is indeed a prime, and the interval $\left(\frac12R_n,R_n\right]$ contains exactly $n$~primes~\cite{sondow}.

In $1919$ Ramanujan proved a result which implies that $R_n$ exists, and he gave the first five Ramanujan primes. (We formulate his result as a theorem and quote him.)

\begin{thm}[Ramanujan] \label{THM: ramanujan}
``Let $\pi(x)$ denote the number of primes not exceeding~$x$. Then $\dotso\pi(x) - \pi\!\left(\frac12x\right) \ge1,2,3, 4, 5, \dotsc$, if $x \ge 2,11,17,29, 41, \dotsc$, respectively.''
\end{thm}
\begin{proof}
This follows from properties of the $\Gamma$-function. See Ramanujan \cite{ramanujan} for details, and Shapiro \cite[Section 9.3B]{shapiro} for an exposition of Ramanujan's idea.
 \end{proof}

The case $R_1=2$ is \emph{Bertrand's Postulate}: for all $x \ge 2$, there exists a prime $p$ with $\frac12x<p\le x$. For $n=1,2,3,\dotsc$, the $n$th Ramanujan prime \cite[Sequence \seqnum {A104272}]{oeis} is
$$R_n=2, 11, 17, 29, 41, 47, 59, 67, 71, 97, 101, 107, 127, 149, 151, 167, 179, 181, 227, 229,233, \dotso.$$

In the present paper, we report progress on three predictions~\cite[Conjectures~$1,2,3$]{sondow} about Ramanujan primes: on bounds, runs, and twins.

In the next section, we sharpen Laishram's theorem that $R_n<p_{3n}$, where $p_n$ denotes the $n$th prime. Namely, we prove the optimal bound that the maximum value of $R_n/p_{3n}$ is $R_5/p_{15} = 41/47$. The proof uses another result of Laishram and a computation of the first $169350$ Ramanujan primes by Noe's fast algorithm. Our first new conjecture follows.

In Section~\ref{SEC: runs}, we present statistics on the length of the longest run of Ramanujan primes among all primes $p<10^n$, for $n\le9$. We pose an open problem on the unexpectedly long runs of non-Ramanujan primes, and make a new conjecture about both types of runs.

In Section~\ref{SEC: twins}, we prove that if the larger of two twin primes is Ramanujan, then its smaller twin is also Ramanujan, and we provide a table of data on the number of twins below $10^n$, again for $n\le9$. We offer several new conjectures and open problems on twin primes.

In Section~\ref{SEC: gaps}, we associate runs of odd Ramanujan primes to certain prime gaps.

The Appendix explains the algorithm for computing Ramanujan primes and includes a \textit{Mathematica} program.

\section {Bounds} \label{SEC: bound}

Here are some estimates for the $n$th Ramanujan prime.

\begin{thm}[Sondow]\label{THM: sondow}
The following inequalities hold:
\begin{equation}
	2n\log2n < p_{2n} < R_n < 4n\log4n < p_{4n} \qquad (n>1). \label{EQ: ineqs}
\end{equation}
Moreover, for every $\epsilon > 0$, there exists $N_0(\epsilon)>0$ such that
\begin{equation}
	R_n < (2+\epsilon)n \log n \qquad (n\ge N_0(\epsilon)). \label{EQ: 1+e}
\end{equation}
In particular, $R_n \sim p_{2n}$ as $n\to\infty$.
\end{thm}
\begin{proof}
Inequalities of Rosser and Schoenfeld for $\pi(x)$, together with \emph{Rosser's theorem} \cite{rosser} that $p_n>n\log n$, lead to \eqref{EQ: ineqs}. The bound \eqref{EQ: 1+e} follows from the Prime Number Theorem. For details, see Sondow \cite{sondow}.
\end{proof}

A prediction~\cite[Conjecture~$1$]{sondow} that  \eqref{EQ: ineqs} can be improved to \mbox{$p_{2n}<R_n < p_{3n}$} has been proved by Laishram.

\begin{thm}[Laishram] \label{THM: laishram}
For all $n\ge1$, we have $R_n < p_{3n}.$
\end{thm}
\begin{proof}
Dusart's inequalities \cite{dusart} for Chebychev's function
$$\theta(x):=\mspace{-10mu}\sum_{\text{prime }p\,\le\, x} \mspace{-10mu}\log p\le\pi(x)\log x$$
lead to an explicit value of $N_0(\epsilon)$ in \eqref{EQ: 1+e}, for each $\epsilon>0$. For details, see Laishram \cite{laishram}.
\end{proof}

Using one of those values and a fast algorithm for computing Ramanujan primes (see the Appendix), we sharpen Theorem~\ref{THM: laishram} by giving an optimal upper bound on $R_n/p_{3n}$, namely, its maximum. (Notice that the rational numbers $R_n/p_{3n}$ are all distinct, because the $p_{3n}$ are distinct primes and $0<R_n/p_{3n}<1$. Thus the maximum occurs at only one value of~$n$.)

\begin{thm} \label{THM: 41/47}
The maximum value of $R_n/p_{3n}$ is
\begin{align*}
	\max_{n\ge1}\frac{R_n}{p_{3n}} &= \frac{R_5}{p_{15}} = \frac{41}{47} = 0.8723\dotso.
\end{align*}
\end{thm}

\begin{proof}
Since $41/47 >  0.8666\dotso =13/15 $, it suffices to show $R_n/p_{3n} < 13/15$ for $n\neq5$.

Set $\epsilon = 3/5$ and substitute $2+\epsilon = 13/5$ into \eqref{EQ: 1+e}. Using Rosser's theorem with $3n$ in place of $n$, we can write the result as
\begin{equation*}
	R_n < \frac{13}{15}\,3n\log n < \frac{13}{15}\,p_{3n}  \qquad(n\ge N_0(3/5)).
\end{equation*}

According to Laishram \cite[Theorem~1]{laishram}, if $0<\epsilon\le1.08$, then $N_0(\epsilon) = (2/\epsilon)^{c/\epsilon}$ in \eqref{EQ: 1+e}, where $c=c(\epsilon)=6$ at $\epsilon= 0.6$. Hence $N_0(3/5) = (10/3)^{10} = 169350.87\dotso$,
and so
\begin{equation*}
	\frac{R_n}{p_{3n}} < \frac{13}{15} \qquad (n>169350).
\end{equation*}

To complete the proof, we compute the first $169350$ Ramanujan primes and then check that $R_n/p_{3n}<13/15$ when $5\neq n\le169350$.
\end{proof}

Similarly, one can show that
\begin{align*}
	\max_{n\neq5}\frac{R_n}{p_{3n}} &= \frac{R_{10}}{p_{30}} = \frac{97}{113} = 0.8584\dotso,\\
	\max_{n\neq5\,\text{\rm or}\,10}\frac{R_n}{p_{3n}} &= \frac{R_{2}}{p_{6}} = \frac{11}{13} = 0.8461\dotso,
\end{align*}
\noindent and so on down towards
$$\lim_{n\to\infty} \frac{R_n}{p_{3n}} = \frac23 = 0.666\dotso.$$

We conclude this section with a related prediction.

\begin{conj} \label{CONJ: R vs p}
For $m=1,2,3,\dotsc,$ let $N(m)$ be given by the following table.

\begin{center}
  \begin{tabular}{|c|c|c|c|c|c|c|c|c|}
  \hline
    $m$      & $1$ &    $2$ &       $3$ &       $4$ &     $5$ &  $6$ &    $7,8,\dotsc,19$ &     $20,21,\dotso$\\ 
    \hline
    $N(m)$ & $1$ & $1245$ & $189$ & $ 189$ & $85$ & $85$ & $10$ & $2$\\ 
    \hline
  \end{tabular}
\end{center}
Then we have
$$\pi(R_{mn}) \le m\pi(R_n) \qquad (n \ge N(m)).$$
Equivalently, if we define the function $\rho$ by $\rho(n):=\pi(R_n)$, so that $R_n=p_{\rho(n)}$, then
$$\rho(mn) \le m\rho(n)  \qquad (n \ge N(m)).$$
\end{conj}

In the cases $m=2,3,\dotsc,20$, the statement has been verified for all $n$ with $R_{mn}<10^9$. The first few values of $\rho(n)$, for $n=1,2,3,\dotsc,$ are \cite[Sequence \seqnum {A179196}]{oeis}
$$\rho(n)=1, 5, 7, 10, 13, 15, 17, 19, 20, 25, 26, 28, 31, 35, 36, 39, 41, 42, 49, 50, 51, 52, 53, \dotsc.$$

Note that Theorems \ref{THM: sondow} and \ref{THM: laishram} imply $2n<\rho(n)<3n$ for all $n>1$, and $\rho(n)\sim2n$ as $n\to\infty$. The latter yields $\rho(mn)\sim2mn\sim m\rho(n)$ as $n\to\infty$, for any fixed $m\ge1$.

\section{Runs} \label{SEC: runs}

Since $p_{2n} < R_n \sim p_{2n}$ as $n\to\infty$, the probability of a randomly chosen prime being Ramanujan is slightly less than~$1/2$, roughly speaking. More precisely, column~$2$ in Table~\ref{TABLE: runs} gives the probability $P_n$ (rounded to $3$ decimal places) that a prime $p<10^n$ is a Ramanujan prime, for $n=1,2,\dotsc,9$.

Let us consider a coin-tossing model. Suppose that a biased coin has probability $P$ of heads. According to Schilling~\cite{schilling}, the expected length E$L_N={\rm E}L_N(P)$ of the longest run of heads in a sequence of $N$ coin tosses is approximately equal to
\begin{equation*}
{\rm E}L_N \approx \frac{\log N}{\log(1/P)} - \left( \frac12 - \frac{\log (1-P)+\gamma}{\log(1/P)}\right),
\end{equation*}
where $\gamma = 0.5772\dotso$ is the Euler-Mascheroni constant. The variance Var$L_N =$ Var$L_N(P)$ is close \cite{schilling} to
\begin{equation*}
{\rm Var}L_N \approx \frac{\pi^2}{6\log(1/P)^2} + \frac{1}{12}
\end{equation*}
``and is quite remarkable for the property that it is essentially constant with respect to'' $N$.

For example, with a fair coin,
\begin{equation}
{\rm E}L_N \approx \frac{\log N}{\log2} - \left( \frac32 - \frac{\gamma}{\log2}\right) = \frac{\log N}{\log2} - 0.667\dotso \qquad\left(P=\frac12\right) \label{EQ: fairexp}
\end{equation}
and
\begin{equation}
{\rm Var}L_N \approx \frac{\pi^2}{6(\log2)^2} + \frac{1}{12} = 3.507\dotso  \qquad\left(P=\frac12\right). \label{EQ: fairvar}
\end{equation}

Schilling points out that by \eqref{EQ: fairvar} ``the standard deviation of the longest run is approximately $({\rm Var}L_N)^{1/2} \approx 1.873$, an amazingly small value. This implies that the length of the longest run is quite predictable indeed; normally it is within about two of its expectation.''

This is roughly true (replacing ``two'' with ``seven'') of the longest run of Ramanujan primes in the sequence of prime numbers below $10^n$ (where $P=P_n\lessapprox1/2$), at least for $n\le9$. But for non-Ramanujan primes (where $P=1-P_n\gtrapprox1/2$), the actual length of the longest run is more than double the expected length, at least for $n=6,7,8,9$. (See Table~\ref{TABLE: runs}, in which the two columns marked ``Actual'' are \cite[Sequences \seqnum {A189993} and \seqnum {A189994}]{oeis}.)

\begin{table}[H]
\begin{center}
\begin{tabular}{|c|c|c|c|c|c|} 
	\hline
	 & Probability $P_n$ of a prime& \multicolumn{4}{|c|}{Length of the longest run below $10^n$ of}   \\
	 & $p<10^n$ & \multicolumn{2}{|c|}{Ramanujan primes} &  \multicolumn{2}{|c|}{non-Ramanujan primes}  \\
	 \cline{3-6}
	 $n$ & being Ramanujan & Expected & Actual & Expected                & Actual  \\
	\hline
	 $1$  & $.250$      & $1$               & $1$                          & $2$                               &  $3$    \\
	 $2$   & $.400$     & $3$              & $2$                          & $5$                                &  $4$    \\
	 $3$   & $.429$     & $6$              & $5$                          & $8$                            &   $7$ \\
	 $4$    & $.455$    & $8$           & $13$                        & $11$                           & $13$\\
	  $5$   & $.465$    & $11$           & $13$                         & $14$                           & $20$ \\
	$6$    & $.471$     & $14$           & $20$                        & $17$                            & $36$ \\
	$7$    & $.476$    & $17$            & $21$                        & $20$                             & $47$ \\
	$8$   & $.479$     & $21$            & $26$                        & $23$                             & $47$ \\
	$9$ & $.482$       & $24$            & $31$                        & $26$                              &  $65$ \\
	\hline
\end{tabular}
\caption{Length of the longest run of (non-)Ramanujan primes below $10^n$.} 
\label{TABLE: runs}
\end{center}
\end{table}

\begin{prob}
Explain the unexpectedly long runs of non-Ramanujan primes among primes $p<10^n$, for $n\ge6$.
\end{prob}

Formula \eqref{EQ: fairexp} suggests the following predictions supported by Table~\ref{TABLE: runs}. They strengthen an earlier prediction \cite[Conjecture~2]{sondow} that arbitrarily long runs of both types exist.

\begin{conj} \label{CONJ: runs}
We have
$$\limsup_{N\to\, \infty} \frac{\text{length of the longest run of Ramanujan primes among primes}\le p_N}{\log N/\log2}\ge1$$
and the same holds true if ``Ramanujan'' is replaced with ``non-Ramanujan''.
\end{conj}

For $n=1,2,\dotsc,$ the first run of $n$ Ramanujan primes begins at
$$2, 67, 227, 227, 227, 2657, 2657, 2657, 2657, 2657, 2657, 2657, 2657, 562871,793487, \dotsc,$$
and the first run of $n$ non-Ramanujan primes at
$$3, 3, 3, 73, 191, 191, 509, 2539, 2539, 5279, 9901, 9901, 9901, 11593, 11593, 55343, 55343,\dotsc,$$
respectively \cite[Sequences \seqnum {A174602} and \seqnum {A174641}]{oeis}.

\section{Twins} \label{SEC: twins}

If $p_n + 2 = p_{n+1}$, then $p_n$ and $p_{n+1}$ are \emph{twin primes}; the smallest are $3$ and~$5$. If $R_n + 2 = R_{n+1}$, then $R_n$ and $R_{n+1}$ are \emph{twin Ramanujan primes}; the smallest are $149$ and $151$.

Given primes $p$ and $q>p$, a necessary condition for them to be twin Ramanujan primes is evidently that
\begin{equation}
 \pi(p) - \pi\!\left(\frac12p\right) + 1 = \pi(q) - \pi\!\left(\frac12q\right). \label{EQ: twin}
\end{equation}

\noindent To see that the condition is not sufficient, even when $p$ and $q$ are \emph{consecutive primes} $p_k$ and $p_{k+1}$, verify \eqref{EQ: twin} for any one of the pairs
\begin{equation}
(p,q)=(p_k,p_{k+1})=(\boldsymbol{11},13),\ (\boldsymbol{47},53),\ (\boldsymbol{67},\boldsymbol{71}),\ (109, 113),\ (137,139), \label{EQ: not suff}
\end{equation}
where Ramanujan primes are in {\bf bold}.

It is less evident that \eqref{EQ: twin}~is a necessary condition for $p$ and $q$ even to be (ordinary) twin primes, but that is not hard to prove \cite[Proposition~1]{sondow}.

\begin{prop} \label{PROP: twin}
If $p$ and $q=p+2$ are twin primes with $p>5$, then~\eqref{EQ: twin} holds.
\end{prop}

The converse is false, even when $p$ and $q$ are consecutive primes both of which are Ramanujan, as the example $(p_{19},p_{20})=(\boldsymbol{67},\boldsymbol{71})=(R_8,R_9)$ shows.

As mentioned, each pair in \eqref{EQ: not suff} consists of consecutive primes $p<q$ satisfying \eqref{EQ: twin}. However, in no pair is $q$ a Ramanujan prime but not~$p$; in fact, such a pair cannot exist.

\begin{prop} \label{PROP: upper consec}
{\rm(i).} If the larger of two twin primes is Ramanujan, then the smaller is also Ramanujan: they are twin Ramanujan primes.

\noindent{\rm(ii).} More generally, given consecutive primes $(p,q)=(p_k,p_{k+1})$ satisfying~\eqref{EQ: twin}, if \mbox{$q=R_{n+1}$}, then $p=R_n$.
\end{prop}
\begin{proof}
Part (i) is (vacuously) true for twin primes $p$ and $q=p+2$ with $p\le5$. For $p>5$ it suffices, by Proposition~\ref{PROP: twin}, to prove part (ii).

Since $q=R_{n+1}$, we have $\pi(x) - \pi\!\left(\frac12x\right)\ge n+1$ when $x\ge q$, and \eqref{EQ: twin} implies that $\pi(p) - \pi\!\left(\frac12p\right)=n$. To prove that $p=R_n$, we have to show that \mbox{$\pi(p-1) - \pi\!\left(\frac12(p-1)\right)<n$}, and that $\pi(x) - \pi(\frac12x)\ge n$ for $x=p+1$, $p+2,\dotsc,q-1$.

If $\ell$ is any prime, then $\pi(\ell-1)+1=\pi(\ell)$ and \mbox{$\pi\!\left(\frac12(\ell-1)\right)= \pi\!\left(\frac12\ell\right)$}, so that the quantity $\pi(y) - \pi\!\left(\frac12y\right)$ increases by $1$ from $y=\ell-1$ to $y=\ell$. Taking $\ell=p$ or $\ell=q$, we infer that \mbox{$\pi(\ell-1) - \pi\!\left(\frac12(\ell-1)\right)=n-1$} or $n$, respectively. As $p$ and $q$ are consecutive primes, it follows that \mbox{$\pi(x)=\pi(q-1)$} and $\pi\!\left(\frac12x\right)\le \pi\!\left(\frac12(q-1)\right)$, for \mbox{$x=p+1$, $p+2,\dotsc,q-1$}, implying \mbox{$\pi(x) - \pi(\frac12x)\ge n$}. This proves the required inequalities.
\end{proof}

Part (i) was conjectured by Noe \cite[Sequence \seqnum {A173081}]{oeis}.

\begin{cor} \label{COR: upper twin}
If we denote
\begin{align*}
\pi_{2,1}(x) :=\ &\#\{\text{pairs of twin primes $\le x\ \colon$one or both are Ramanujan}\}, \\ 
\pi_{2,2}(x) :=\ &\#\{\text{pairs of twin primes $\le x\ \colon$both are Ramanujan}\},
\intertext{then for all $x \ge 0$ we have the equalities}
\pi_{2,1}(x) =\ &\#\{\text{pairs of twin primes $\le x\ \colon$the smaller is Ramanujan}\},\\
\pi_{2,2}(x) =\ &\#\{\text{pairs of twin primes $\le x\ \colon$the larger is Ramanujan}\}.
\end{align*}
\end{cor}
\begin{proof}
By Proposition \ref{PROP: upper consec} part (i), given twin primes $p$ and $p+2$, if $p+2=R_{n+1}$, then $p=R_n$. The corollary follows.
\end{proof}

Table~\ref{TABLE: twins} gives some figures (see \cite[Sequences \seqnum {A007508}, \seqnum {A173081}, \seqnum {A181678}]{oeis}) on
$$\pi_2(x) := \#\{\text{pairs of twin primes} \le x\},$$
$\pi_{2,1}(x),\pi_{2,2}(x)$, and their ratios. Proposition~\ref{PROP: twin} and Corollary~\ref{COR: upper twin} will help to explain why many values of the ratios are greater than might be expected a priori.

\begin{table}[H]
\begin{center}
\begin{tabular}{|c|r|r|r|c|c|c|} 
	\hline
	 & \multicolumn{6}{|c|}{$\pi_* = \pi_*(10^n)$}  \\
	\hline
	 $n$ & $\pi_2\hspace{1.2em}$     & $\pi_{2,1}\hspace{1em}$ & $\pi_{2,2}\hspace{1em}$ & $\pi_{2,1}/\pi_2$ & $\pi_{2,2}/\pi_2$ & $\pi_{2,2}/\pi_{2,1}$   \\
	\hline
	 $1$& $2$             & $0$               & $0$                 & $0$                               & $0$                                 &  --    \\
	 $2$& $8$             & $6$               & $0$                 & $.750$                          & $0$                                 &  $0$    \\
	 $3$& $35$           & $28$             & $10$              & $.800$                          & $.286$                        &   $.357$ \\
	 $4$& $205$         & $167$           & $73$             & $.815$                        & $.356$                           & $.437$\\
	  $5$& $1224$      & $694$           & $508$          & $.788$                         & $.415$                           & $.527$ \\
	$6$& $8169$        & $6305$         & $3468$        & $.772$                        & $.425$                            & $.550$ \\
	$7$& $58980$      & $45082$       & $25629$     & $.764$                        & $.434$                             & $.568$ \\
	$8$ & $440312$   & $335919$    & $194614$   & $.763$                        & $.442$                             & $.579$ \\
	$9$ & $3424506$ & $2605867$ & $1537504$ & $.761$                        & $.449$                              &  $.590$ \\
	\hline
\end{tabular}
\caption{Counting three types of pairs of twin primes below $10^n$.} 
\label{TABLE: twins}
\end{center}
\end{table}

The probability that two randomly chosen primes $p$ and $q$ are both Ramanujan is slightly less than \mbox{$1/2\times1/2=1/4$}, roughly speaking. The probability increases if $p$ and $q$ are twin primes, because then Proposition~\ref{PROP: twin} guarantees that the necessary condition \eqref{EQ: twin} holds.

For that reason, and based on the first $1000$ Ramanujan primes, it was predicted \cite[Conjecture~$3$]{sondow} that more than $1/4$ of the twin primes up to~$x$ are twin Ramanujan primes, if $x\ge571$. This is borne out for $x=10^n$, with $3\le n\le9$, by Table~\ref{TABLE: twins}. It shows that the prediction can be improved to $\pi_{2,2}(x)/\pi_2(x)>2/5$, for $x\ge10^5$.

Corollary~\ref{COR: upper twin} implies that whether a twin prime pair is counted in $\pi_{2,1}(x)$ or $\pi_{2,2}(x)$ depends on only one of the two primes being Ramanujan. This suggests that the ratios $\pi_{2,1}(x)/\pi_2(x)$ and $\pi_{2,2}(x)/\pi_2(x)$ should approach $1/2$ as $x$ tends to infinity.

We conclude this section with these and other conjectures based on our results and on Table~\ref{TABLE: twins}, as well as with two more open problems.

\begin{conj} \label{CONJ: 1/4}
For all $x \ge 10^5$, we have
\begin{align*} \label{EQ: ratios}
\frac{\pi_{2,1}(x)}{\pi_2(x)} &< \frac45, \qquad
\frac{\pi_{2,2}(x)}{\pi_2(x)} > \frac25, \qquad
\frac{\pi_{2,2}(x)}{\pi_{2,1}(x)} > \frac{2/5}{4/5}= \frac12.
\end{align*}
\end{conj}

\begin{conj} \label{CONJ: lim}
If $\pi_2(x)\to\infty$ as $x\to\infty$, then $\pi_{2,1}(x) \sim \pi_{2,2}(x) \sim \frac12\pi_2(x)$.
\end{conj}

Recall Brun's famous theorem \cite{brun} that the series of reciprocals of the twin primes converges or is finite (unlike the series of reciprocals of all the primes, which Euler showed diverges). Its sum \cite[Sequence \seqnum {A065421}]{oeis} is \emph{Brun's constant} $B_2$,
$$B_2:=\left(\frac13 + \frac{1}{5}\right)+\left(\frac15 + \frac{1}{7}\right) + \left(\frac{1}{11}+ \frac{1}{13}\right)+ \left(\frac{1}{17}+ \frac{1}{19}\right)+\dotsb \stackrel{?}{=}
1.9021605\dotso.$$
Here $\stackrel{?}{=}$ means that the value of $B_2$ is conditional ``on heuristic considerations about the distribution of twin primes'' (Ribenboim \cite[p.~201]{rib89}).

\begin{prob}
Compute the analogous constant $B_{2,1}$ for twin primes at least one of which is Ramanujan,$$B_{2,1}:= \left(\frac{1}{11}+ \frac{1}{13}\right)+ \left(\frac{1}{17}+ \frac{1}{19}\right)+ \left(\frac{1}{29}+ \frac{1}{31}\right)+ \left(\frac{1}{41}+ \frac{1}{43}\right)+\dotsb.$$
\end{prob}

The numbers $11, 17, 29,$ $41,\dotso$ \cite[Sequence \seqnum {A178128}]{oeis} are the lesser of twin primes if at least one is Ramanujan. By Corollary~\ref{COR: upper twin}, that is the same as the lesser of twin primes if it is Ramanujan.

\begin{prob}
Compute the analogous constant $B_{2,2}$ for twin Ramanujan primes,$$B_{2,2}:= \left(\frac{1}{149}+ \frac{1}{151}\right)+ \left(\frac{1}{179}+ \frac{1}{181}\right)+ \left(\frac{1}{227}+ \frac{1}{229}\right)+ \left(\frac{1}{239}+ \frac{1}{241}\right)+\cdots.$$
\end{prob}

The numbers $149, 179, 227,239,\dotso$ \cite[Sequence \seqnum {A178127}]{oeis} are the lesser of twin Ramanujan primes.

\section{Prime gaps} \label{SEC: gaps}

Let us say that there is a \emph{prime gap from $a$ to~$b\ge a$} if none of the numbers $a,a+1,a+2,\dotsc, b$ is prime. Given a run of $r$ odd Ramanujan primes starting at $p$, we can associate to it a prime gap of length at least $r$ starting at $\frac12(p+1)$.

\begin{prop} \label{PROP: gaps}
{\rm(i).} If $p=R_n$ is odd, then the integer $\frac12(p + 1)$ is not prime.

\noindent{\rm(ii).} More generally, given a run of $r\ge1$ odd Ramanujan primes from $p=R_n=p_k$ to $q=R_{n+r-1}=p_{k+r-1}$, there is a prime gap from $\frac12(p + 1)$ to $\frac12(q + 1)$.

\noindent{\rm(iii).} Parts {\rm (i)} and {\rm (ii)} are sharp in the sense that, for \emph{certain} runs of Ramanujan primes $p$ to $q$, both $\frac12(p + 1)-1$ and $\frac12(q + 1)+1$ are prime numbers.

\noindent{\rm(iv).} But in the case $r=2$, if $p$ and $q$ are \emph{twin} Ramanujan primes, then the prime gap from $\frac12(p + 1)$ to $\frac12(q+1)$ always lies in a longer prime gap of length $5$ or more.
\end{prop}
\begin{proof}(i). Since $p=R_n$ is odd, $\pi(p)=\pi(p+1)$, and the quantity $\pi(x) - \pi\!\left(\frac12x\right)$ does not decrease from $x=p$ to $x=p+1$. Hence $\pi\!\left(\frac12p\right)\ge\pi\!\left(\frac12(p+1)\right)$, and so $\frac12(p + 1)$ is not prime.

\noindent(ii). By (i), the case $r=1$ holds. Taking $r=2$, let $p=R_n=p_k$ and $q=R_{n+1}=p_{k+1}$ be odd. By (i), neither $\frac12(p + 1)$ nor $\frac12(q + 1)$ is prime. If an integer $i$ lies strictly between them, then the oddness of $p$ and $q$ implies $p+1< j:=2i-1< q-1$. Since $p=p_k$ and $q=p_{k+1}$, we have $\pi(p)=k=\pi(j+1)$. As $p=R_n$ and $q=R_{n+1}$, it follows that $\pi(x) - \pi\!\left(\frac12x\right)$ does not decrease from $x=p$ to $x=j+1$. Hence $\pi\!\left(\frac12p\right)\ge\pi\left(\frac12(j+1)\right) =  \pi(i)$, and so $i$ is also not prime. This proves (ii) for runs of length $2$.

The general case follows easily by induction on $r$. Namely, given a run of length $r>2$ from $R_n=p_k$ to $R_{n+r-1}=p_{k+r-1}$, break it into a run of length $2$ from $R_n=p_k$ to $R_{n+1}=p_{k+1}$, concatenated with a run of length $r-1$ from $R_{n+1}=p_{k+1}$ to $R_{n+r-1}=p_{k+r-1}$.

\noindent(iii). For $r=1$, the composite number $\frac12(R_2 + 1)=\frac12(11+1)=6$ lies between the primes $5$ and~$7$. For an example with $r>1$, take the run $(R_{293}, R_{294})=(4919,4931)=(p_{657}, p_{658})$ of length $r=2$. It is associated to the prime gap from
$\frac12(R_{293} + 1) = 2460$ to $\frac12(R_{294} + 1)=2466$,
which is bounded by the primes $2459$ and $2467$.

\noindent(iv). Since $p>3$ and $q$ are twin primes, $(p,q)=(6k-1,6k+1)$ for some $k$. If $k=2i$ is even, then $\left(\frac12(p + 1),\frac12(q+1)\right)=(6i,6i+1)$ lies in the prime gap from $6i$ to $6i+4$.

Now assume that $k=2i+1$ is odd. Then $\left(\frac12(p + 1),\frac12(q+1)\right)=(6i+3,6i+4)$ will lie in a prime gap from $6i+2$ to $6i+6$, unless $6i+5$ is prime. But if $6i+5=\frac12(q+3)$ were prime, then, since $q+2=6k+3$ is not prime, $\pi(x) - \pi\!\left(\frac12x\right)$ would decrease from $x=q$ to $x=q+3$, contradicting the fact that $q$ is a Ramanujan prime. This completes the proof.
\end{proof}

For part (iii), the first ``sharp'' example of a run of length $r=1,2,\dotsc,11$ begins at the Ramanujan prime
$$11, 4919, 1439, 7187, 37547, 210143, 3376943, 663563, 4429739, 17939627, 12034427,$$
respectively \cite[Sequence \seqnum {A177804}]{oeis}. An example of part (iv) is the prime gap associated to the twin Ramanujan primes $R_{14}=149$ and $R_{15}=151$, which lies in the larger prime gap
$$74,\frac12(R_{14} + 1) = 75,\frac12(R_{15} + 1)=76,77,78.$$

\section{Appendix on the algorithm}

To compute a range of Ramanujan primes $R_i$ for $1 \le i \le n$, we
perform simple calculations in each interval $\left(k/2,k\right]$ for
$k=1,2,\ldots,p_{3n}-1$. To facilitate the calculation, we use a counter $s$ and a list
$L$ with $n$ elements $L_i$. Initially, $s$ and all $L_i$ are set to zero. They are updated as each interval is processed.

After processing an interval, $s$ will be equal to the number of primes in
that interval, and each $L_i$ will be equal either to the greatest index of the
intervals so far processed that contain exactly $i$ primes, or to zero if no interval
having exactly $i$ primes has yet been processed.

Having processed interval $k-1$, to find the number of primes in interval $k$ we perform two operations: add 1 to $s$ if
$k$ is prime, and subtract 1 from $s$ if $k/2$ is prime. We then update the $s$th element of the list
to $L_s=k$, because now $k$ is the largest index of all intervals processed that contain
exactly $s$ primes.

After all intervals have been processed, the list $R$ of
Ramanujan primes is obtained by adding 1 to each element of the list
$L$.

These ideas are captured in the following \textit{Mathematica} program for finding the first 169350 Ramanujan primes. 

\begin{center}
  \begin{tabular}{cccclccc}
      $$ &    & $$ &    $$ &       $\texttt{nn = 169350; }$ &       $$ &       \\ 
         $$ &  & $$ &    $$ &       $\texttt{L = Table[0, \{nn\}]; }$ &       $$ &    \\ 
       $$ &    & $$ &    $$ &       $\indent\texttt{s = 0; }$ &       $$ &      \\ 
       $$ &    & $$ &    $$ &       $\indent\texttt{Do[ }$ &       $$ &       \\ 
       $$ &    & $$ &    $$ &       $\hspace{.6cm}\texttt{If[PrimeQ[k], s++]; }$ &       $$ &      \\ 
        $$ &   & $$ &    $$ &       $\hspace{.6cm}\texttt{If[PrimeQ[k/2], s--]; }$ &       $$ &      \\ 
       $$ &    & $$ &    $$ &       $\hspace{.6cm}\texttt{If[s < nn, L[[s+1]] = k], }$ &       $$ &      \\ 
        $$ &   & $$ &    $$ &       $\hspace{.6cm}\texttt{\{k, Prime[3*nn]-1\}]; }$ &       $$ &     \\ 
        $$ &   & $$ &    $$ &       $\texttt{R = L + 1 }$ &       $$ &     $$ &   
  \end{tabular}
\end{center}

Although it is adequate for computing a modest number of
them, to compute many more requires a speedup of several orders of magnitude. That can be achieved
by using a lower-level programming language and generating prime numbers
via a sieve. With this speedup we computed all Ramanujan primes
below $10^9$ in less than three minutes on a 2.8 GHz Pentium 4 computer.

\section{Acknowledgment}

The authors thank Steven Finch for suggesting Schilling's paper \cite{schilling}. We are grateful to Professor Steven J. Miller's undergraduate students Nadine Amersi, Olivia Beckwith, and Ryan Ronan for pointing out and diagnosing mistakes (corrected here) in the ``Expected'' columns of Table~\ref{TABLE: runs} in the published version.

\bigskip
\hrule
\bigskip

\noindent 2010 {\it Mathematics Subject Classification}: 
Primary 11A41.     

\noindent \emph {Keywords:} prime gap, Ramanujan prime, twin prime.   

\bigskip
\hrule
\bigskip

\noindent (Concerned with sequences
\seqnum {A007508},
\seqnum {A065421},
\seqnum {A104272},
\seqnum {A173081},
\seqnum {A174602},
\seqnum {A174641},
\seqnum {A177804},
\seqnum {A178127},
\seqnum {A178128},
\seqnum {A179196},
\seqnum {A181678},
\seqnum {A189993},
and
\seqnum {A189994}.)


\begin{thebibliography}{99}

\bibitem{brun}V.~Brun, La s\'erie $ \frac 15 + \frac 17 + \frac {1}{11} + \frac {1}{13} + \frac {1}{17} + \frac {1}{19} + \frac {1}{29} + \frac {1}{31} + \frac {1}{41} + \frac {1}{43} + \frac {1}{59} + \frac {1}{61} +\dotsb$ o\`u les d\'enominateurs sont ``nombres premiers jumeaux'' est convergente ou finie, \textit{Bull. Soc. Math. France} \textbf{43} (1919), 100--104, 124--128.

\bibitem{dusart}P.~Dusart, In\'egalit\'es explicites pour $\psi(X)$, $\theta(X)$, $\pi(X)$ et les nombres premiers, \textit{C. R. Math. Acad. Sci. Soc. R. Can.} \textbf{21} (1999), 53--59.

\bibitem{laishram}S.~Laishram, On a conjecture on Ramanujan primes, \textit{Int. J. Number Theory} \textbf{6} (2010), 1869--1873; also available at \url{http://www.isid.ac.in/~shanta/PAPERS/RamanujanPrimes.pdf}.

\bibitem{ramanujan}S.~Ramanujan, A proof of Bertrand's postulate, \textit{J. Indian Math. Soc.} \textbf{11} (1919), 181--182; also available at \url{http://www.imsc.res.in/~rao/ramanujan/CamUnivCpapers/Cpaper24/}.

\bibitem{rib89} P.~Ribenboim, \emph{The Book of Prime Number Records}, 2nd. ed., Springer-Verlag, New York, 1989.

\bibitem{rosser}J.~B.~Rosser, The $n$th prime is greater than $n \ln n$, \textit{Proc. London Math. Soc.} \textbf{45} (1938), 21--44.

\bibitem{schilling}M.~F.~Schilling, The longest run of heads, \textit{College Math. J.} \textbf{21} (1990), 196--207; also available at \url{http://users.eecs.northwestern.edu/~nickle/310/2010/headRuns.pdf}.

\bibitem{shapiro} H.~N.~Shapiro, \emph{Introduction to the Theory of Numbers}, Wiley, New York, 1983; reprinted by Dover, Mineola, NY, 2008.

\bibitem {oeis}N.~J.~A.~Sloane, {\em \htmladdnormallink {The On-Line Encyclopedia of Integer Sequences}{http://oeis.org}}, published electronically at \url{http://oeis.org}, 2011.

\bibitem{sondow}J.~Sondow, Ramanujan primes and Bertrand's postulate, \textit{Amer. Math. Monthly} \textbf{116} (2009), 630--635; also available at \url{http://arxiv.org/abs/0907.5232}.

\end{thebibliography}
\end{document}